\title{\vspace{-0.6cm} On a $d$-degree Erd\H os--Ko--Rado Theorem}

\documentclass[12pt]{article}

\usepackage{amsmath, amssymb, amsthm, url, tikz, verbatim, color}

\usepackage{blkarray}

\newcommand{\genlegendre}[4]{%
	\genfrac{(}{)}{}{#1}{#3}{#4}%
	\if\relax\detokenize{#2}\relax\else_{\!#2}\fi
}

\newtheorem{theorem}{Theorem}[section]

\newtheorem{lemma}[theorem]{Lemma}

\newtheorem{conjecture}[theorem]{Conjecture}

\usetikzlibrary{calc}
\usetikzlibrary{patterns}
\usetikzlibrary{decorations.markings}
\usetikzlibrary{arrows,shapes.geometric,%
	decorations.pathreplacing,shapes,shadows}

\author{
  Hao Huang \thanks{Department of Mathematics, National University of Singapore. Email: huanghao@nus.edu.sg. Research supported in part by a start-up grant at NUS and an MOE Academic Research Fund (AcRF) Tier 1 grant.}
	\and Yi Zhang \thanks{Department of Mathematics, Beijing University of Posts and Telecommunications. Email: shouwangmm@sina.com. Research supported by Fundamental Research Funds for the Central Universities and Innovation Foundation of BUPT
for Youth under Grant No. 2023RC49 and National Natural Science Foundation of China under Grant No. 11901048 and
12071002.}
}
\date{}

\begin{document}
\maketitle
\abstract{A family of subsets $\mathcal{F}$ is intersecting if $A \cap B \neq \emptyset$ for any $A, B \in \mathcal{F}$. In this paper, we show that for given integers $k > d \ge 2$ and $n \ge 2k+2d-3$, and any intersecting family $\mathcal{F}$ of $k$-subsets of $\{1, \cdots, n\}$, there exists a $d$-subset of $[n]$ contained in at most $\binom{n-d-1}{k-d-1}$ subsets of $\mathcal{F}$. This result, proved using spectral graph theory, gives a $d$-degree generalization of the celebrated Erd\H os--Ko--Rado Theorem, improving a theorem of Kupavskii.}

\section{Introduction}

Denote by $[n]$ the set $\{1, 2, \cdots, n\}$, and $\binom{[n]}{k}$ the family of all the $k$-subsets of $[n]$. The celebrated Erd\H os--Ko--Rado (EKR) Theorem \cite{ekr} states that for $n \ge 2k$, an intersecting family $\mathcal{F} \subset \binom{[n]}{k}$ has at most $\binom{n-1}{k-1}$ subsets. The upper bound is attained by an {\it $1$-star}, i.e. the family of all the size-$k$ subsets of $[n]$ that contain a fixed element. Moreover, when $n \ge 2k+1$, this construction is unique up to isomorphism.  Generalizations and analogues of the EKR  Theorem have been studied over the years, see for example \cite{FrGr89} for different proofs and \cite{DeFr83} for a survey. 

One may view $\mathcal{F} \subset \binom{[n]}{k}$ as a $k$-uniform $n$-vertex hypergraph and consider its hypergraph-theoretic properties. Given a subset $S \subset [n]$, the {\it $S$-degree} of $\mathcal{F}$ is defined as $d_S(\mathcal{F})=|\{T: S \subset T \in \mathcal{F}\}|$. In other words, $d_S(\mathcal{F})$ counts the number of subsets containing $S$ in $\mathcal{F}$. For a fixed integer $0 \le d \le k$, the {\it minimum $d$-degree} of $\mathcal{F}$, denoted by $\delta_d(\mathcal{F})$, is the minimum of $d_S(\mathcal{F})$, over all $|S|=d$. For instance, $\delta_0(\mathcal{F})$ is simply the number of edges of $\mathcal{F}$, and $\delta_1(\mathcal{F})$ is the usual minimum vertex degree of $\mathcal{F}$. In this language, The EKR Theorem determines the maximum of $\delta_0(\mathcal{F})$ among all intersecting $k$-uniform $n$-vertex hypergraphs.

Motivated by the study of Dirac-type problems in hypergraphs (see \cite{RRsurvey, zhao_survey} for two comprehensive surveys), Huang and Zhao \cite{degree_ekr} established the following degree analogue of the EKR Theorem: for $n \ge 2k+1$ and every intersecting family $\mathcal{F} \subset \binom{[n]}{k}$, $\delta_1(\mathcal{F}) \le \binom{n-2}{k-2}$. Both the range $n \ge 2k+1$ and the bound $\binom{n-2}{k-2}$ given by $1$-stars are best possible. Actually, for most values of $k$, there are regular intersecting subfamilies of $\binom{[2k]}{k}$ of size $\binom{2k-1}{k-1}$, as shown by Ihringer and Kupavskii \cite{ihr-kup}. Subsequently, Frankl and Tokushige \cite{frankl-tokushige} gave a combinatorial proof of the Huang--Zhao Theorem for $n \ge 3k$, based on an earlier result by Frankl \cite{frankl_max} which provides sharp upper bounds on the size
of intersecting families with certain maximum degree. By analyzing their technique more carefully, one can actually get the proof work for $n\ge 2k+3$, but it seems very hard to obtain the exact range $n \ge 2k+1$ via combinatorial means. Later on, Frankl, Han, Huang and Zhao \cite{degree_hm} also proved a degree version of the Hilton--Milner theorem.

It is not hard to observe that by induction on $n$, the Huang--Zhao Theorem implies the classical EKR Theorem. A natural question arises: for $d \ge 2$, does the $1$-star still  maximize the minimum $d$-degree $\delta_d(\mathcal{F})$, among all intersecting families $\mathcal{F}$? Kupavskii \cite{kup} confirmed this speculation for sufficiently large $n$. He proved that if $n \ge 2k+ \frac{3d}{1-d/k}$ and  $1 \le d < k$, then $\delta_d(\mathcal{F}) \le \binom{n-d-1}{k-d-1}$,  Note that when $d$ and $k$ are very close, the lower bound on $n$ can be quadratic in $d$ or $k$. In this paper, we use techniques from spectral graph theory to significantly improve this range. Our main result is the following.

\begin{theorem}\label{thm_main}
For positive integers $n, k, d$ satisfying $k > d \ge 2$ and $n \ge 2k+2d-3$ , every intersecting family $\mathcal{F}$ of $k$-subsets of $[n]$ must have 
$$\delta_d(\mathcal{F})\le \binom{n-d-1}{k-d-1}.$$
\end{theorem}
Note that when $d=2$, the range $n \ge 2k+1$ in Theorem \ref{thm_main} is the best we could hope for, since the conclusion might fail for $n=2k$. For example, for $(n, k, d)=(6, 3, 2)$, the unique $2$-$(6,3,2)$-design $\{6, 1, 2\}$, $\{6, 1, 3\}$, $\{6, 2, 4\}$, $\{6, 3, 5\}$, $\{6, 4, 5\}$, $\{1, 2, 5\}$, $\{1, 3, 4\}$, $\{1, 4, 5\}$, $\{2, 3, 4\}$, $\{2,3, 5\}$ gives an intersecting family with $\delta_2(\mathcal{F})=2>\binom{6-3}{3-3}$.

The rest of this paper is organized as follows. In the next section, we introduce some notations used throughout the paper, and sketch an outline of the proof. Our main result, Theorem \ref{thm_main}, will be proved in Section \ref{sec_main}, after two key lemmas are established. To ensure a smoother reading experience, we leave the detailed proofs of several unsurprising identities and inequalities to Appendix \ref{app}. Some remarks and open problems are discussed in Section \ref{sec_concluding}.

\section{Notations and preparations} \label{sec_pre}
For $1 \le i \le k-1$, we say that a family of $k$-subsets is an {\it $i$-star}, if it consists exactly of every $k$-set containing a fixed $i$-set. The {\it indicating vector} of a family $\mathcal{F} \subset \binom{[n]}{k}$, denoted by $\vec{1}_{\mathcal{F}}$, is a $\binom{n}{k}$-dimensional binary vector, whose coordinates correspond to subsets $S \in \binom{[n]}{k}$, and whose $S$-th coordinate is $1$ if $S \in \mathcal{F}$ and $0$ otherwise.

We denote by $KG(n, k)$ the Kneser graph with parameter $n \ge 2k$. Its vertex set consists of all the $k$-subsets of $[n]=\{1, \cdots, n\}$. Two vertices are adjacent if and only if their corresponding $k$-sets are disjoint. Let $A$ be its adjacency matrix. It is well-known that $A$ has the following eigenspace decomposition (see for example \cite{lovasz} or \cite{wilson}):
$$\mathbb{R}^{\binom{n}{k}} = E_0 \oplus E_1 \cdots \oplus E_k,$$ such that
\begin{itemize}
\item $\dim(E_i)=\binom{n}{i}-\binom{n}{i-1}$.
\item For any vector $\vec{v} \in E_i$, $A\vec{v}=(-1)^i\binom{n-k-i}{k-i} \vec{v}$.
\item The subspace $F_i=E_0 \oplus \cdots \oplus E_i$ is an $\binom{n}{i}$-dimensional subspace, spanned by the indicating vectors of the $\binom{n}{i}$ distinct $i$-stars. In other words, $E_i$ is the orthogonal complement of $F_{i-1}$ in $F_i$.
\end{itemize}
Let $\vec{h}=\vec{1}_{\mathcal{F}} \in \mathbb{R}^{\binom{n}{k}}$ be the indicating vector of the intersecting family $\mathcal{F}$. For $i=0, \cdots, k$, we consider the projection of $\vec{h}$ onto the subspace $E_i$ and denote it by $\vec{h}_i$. We have 
$$\vec{h}=\vec{h}_0 + \cdots + \vec{h}_k,$$
$$A\vec{h}_i=(-1)^i \binom{n-k-i}{k-i}\vec{h}_i.$$
Clearly for $i \neq j$, the vectors $\vec{h}_i$ and $\vec{h}_j$ are orthogonal.

In the next section, we will prove two inequalities involving the lengths of $\vec{h}_i$'s. The first one follows from a Hoffman bound type argument, under the assumption that $\mathcal{F}$ is intersecting. The second inequality is derived directly from information on the minimum $d$-degree $\delta_{d}(\mathcal{F})$. We will show that a carefully chosen linear combination of these two inequalities yields a contradiction to the Erd\H os--Ko--Rado Theorem which states that $|\mathcal{F}| \le \binom{n-1}{k-1}$ whenever $n \ge 2k$.
\section{The main theorem}\label{sec_main}
The first lemma compares two binomial coefficients, helping us later to bound the small eigenvalues of the Kneser graph $KG(n, k)$.
\begin{lemma}\label{lem_ineq}
For $n \ge 2k+1$ and $k>d\ge 1$, we have 
$$\frac{k-d}{n-k}\binom{n-k}{k} \ge \binom{n-k-d-1}{k-d-1}.$$
Furthermore, the equality only holds when $n=2k+1$.
\end{lemma}
\begin{proof}
We have
\begin{align*}
\frac{\binom{n-k-d-1}{k-d-1}}{\binom{n-k}{k}}&=\frac{(k-d)\cdots k}{(n-k-d)\cdots (n-k)}\\&=\frac{k-d}{n-k}\cdot \left(\frac{k-d+1}{n-k-d} \cdot \frac{k-d+2}{n-k-d+1}\cdots \frac{k}{n-k-1}\right). 
\end{align*}
Since $n \ge 2k+1$, each term in the product inside the parentheses is at most $1$. Thus the above expression is at most $\frac{k-d}{n-k}$. The equality only holds only when these terms are all $1$, which occurs when $n=2k+1$. 
\end{proof}
The next lemma considers the projections of the indicating vector $\vec{h}$ of $\mathcal{F}$ onto the eigenspaces of the Kneser graph $KG(n, k)$. 
\begin{lemma}\label{lem1}
For an intersecting family $\mathcal{F} \subset \binom{[n]}{k}$ with $n \ge 2k+1, k>d \ge 1$, we have
$$0 \ge -\frac{k-d}{n-k}\binom{n-k}{k}|\mathcal{F}|+\sum_{i=0}^d \left(\frac{k-d}{n-k}\binom{n-k}{k}+(-1)^i\binom{n-k-i}{k-i}\right)\|\vec{h}_i\|^2.$$
\end{lemma}
\begin{proof}
Since $\mathcal{F} \subset \binom{[n]}{k}$ is intersecting, it induces an independent set in the Kneser graph $KG(n, k)$. Therefore from the eigenspace decomposition in the previous section, we have
\begin{align*}
0&=\vec{h}^T A \vec{h}=\sum_{i=0}^k (-1)^i \binom{n-k-i}{k-i} \|\vec{h}_i\|^2\\
& =\sum_{i=0}^d (-1)^i \binom{n-k-i}{k-i} \|\vec{h}_i\|^2 + \sum_{i=d+1}^k (-1)^i \binom{n-k-i}{k-i} \|
\vec{h}_i\|^2\\
& \ge \sum_{i=0}^d (-1)^i \binom{n-k-i}{k-i} \|\vec{h}_i\|^2  - \binom{n-k-d-1}{k-d-1} \sum_{i=d+1}^k \|
\vec{h}_i\|^2\\
&\ge \sum_{i=0}^d (-1)^i \binom{n-k-i}{k-i} \|\vec{h}_i\|^2  - \frac{k-d}{n-k}\binom{n-k}{k} \sum_{i=d+1}^k \|\vec{h}_i\|^2\\
&=-\frac{k-d}{n-k}\binom{n-k}{k}|\mathcal{F}|+\sum_{i=0}^d \left(\frac{k-d}{n-k}\binom{n-k}{k}+(-1)^i\binom{n-k-i}{k-i}\right)\|\vec{h}_i\|^2.
\end{align*}
The first inequality uses $\binom{n-k-i}{k-i} \le \binom{n-k-d-1}{k-d-1}$ for all $i=d+1, \cdots, k$ and $n \ge 2k$. The second inequality is a direct consequence from the comparison in Lemma \ref{lem_ineq}. The last identity follows from $|\mathcal{F}|=\|\vec{h}\|^2=\sum_{i=0}^k \|\vec{h}_i\|^2.$
\end{proof}

The next lemma is based on the assumption of having large minimum $d$-degree. It does not require $\mathcal{F}$ to be intersecting.

\begin{lemma}\label{lem2}
Suppose $\mathcal{F}$ is a family of $k$-subsets of $[n]$, such that $\delta_d(\mathcal{F}) \ge \binom{n-d-1}{k-d-1}$, then the following inequality holds:
\begin{align*}
0 &\le \sum_{j=0}^d(-1)^j\binom{k-j}{d-j}\binom{n-d-j}{d-j}\binom{n-d-j}{k-d} \|\vec{h}_j\|^2 \\
&~~~ -2\binom{n-d-1}{k-d-1}\binom{n-d}{d}\binom{k}{d}|\mathcal{F}|+\binom{n-d-1}{k-d-1}^2\binom{n}{d}\binom{n-d}{d}.
\end{align*}
Furthermore,  the above inequality becomes strict if $\delta_{d}(\mathcal{F})>\binom{n-d-1}{k-d-1}$.
\end{lemma}
\begin{proof}
Since $d_S=d_S(\mathcal{F}) \ge \binom{n-d-1}{k-d-1}$ for all $|S|=d$, we have that for all $S, T \in \binom{[n]}{d}$,
$$\left(d_T-\binom{n-d-1}{k-d-1}\right)\left(d_S-\binom{n-d-1}{k-d-1}\right) \ge 0.$$
Summing over those (ordered) disjoint pairs, we have
\begin{align*}
0 &\le \sum_{S, T: |S|=|T|=d, S \cap T=\emptyset} \left(d_T-\binom{n-d-1}{k-d-1}\right)\left(d_S-\binom{n-d-1}{k-d-1}\right)\\
&=\vec{D}_d^T M \vec{D}_d - 2\binom{n-d-1}{k-d-1}\binom{n-d}{d} \sum_{S:|S|=d} d_S + \binom{n-d-1}{k-d-1}^2\binom{n}{d}\binom{n-d}{d}.
\end{align*}
Here $M$ is the $\binom{n}{d}$ by $\binom{n}{d}$ matrix with rows and columns indexed by $d$-subsets of $[n]$, and $M_{S, T}$ equals $1$ if $S \cap T=\emptyset$ and $0$ otherwise. $\vec{D}_d$ stands for the $\binom{n}{d}$-dimensional vector, whose coordinates record all the $d$-degrees of $\mathcal{F}$. By double counting, we have $\sum_{S: |S|=d} d_S=\binom{k}{d}|\mathcal{F}|$. So it remains to rewrite the quadratic part $\vec{D}_d^T M \vec{D}_d$ as a linear combination of $\|\vec{h}_i\|^2$'s. 

For $0 \le i \le j \le k$, we define the {\it $(i, j)$-inclusion matrix} $W_{i, j}$ to be the matrix with rows indexed by all $i$-subsets of $[n]$, columns indexed by all $j$-subsets of $[n]$, and $(W_{i,j})_{S,T}=1$ if $S \subset T$, and $0$ otherwise. Then we have $\vec{D}_d=W_{d,k}\vec{h}.$ Similarly, we also define the {\it $(i,j)$-disjointness matrix} $\overline{W}_{i, j}$ to have $(S, T)$-entry $1$ if $S \cap T=\emptyset$ and $0$ otherwise. With this notation, the matrix $M$ in the quadratic form is the same as $\overline{W}_{d,d}$.

Now we can rewrite the quadratic term $\vec{D}_d^T M \vec{D}_d$ as $\vec{h}^T(W_{d,k}^T\overline{W}_{d,d}W_{d,k}) \vec{h}$. It is not hard to see that $W_{d,k}^T\overline{W}_{d,d}W_{d,k}$ is an $\binom{n}{k} \times \binom{n}{k}$ matrix whose $(S, T)$-entry only depends on $|S \cap T|$. Thus it has the same eigenspace decomposition $E_0 \oplus \cdots \oplus E_k$ as that of the adjacency matrix of $KG(n, k)$, and
\begin{align*}
W_{d,k}^T\overline{W}_{d,d}W_{d,k}&=W_{d,k}^T\left(\sum_{i=0}^d (-1)^i W_{i,d}^T W_{i,d}\right)W_{d,k}=\sum_{i=0}^d (-1)^i (W_{i,d}W_{d,k})^TW_{i,d}W_{d,k}\\
&=\sum_{i=0}^d (-1)^i \binom{k-i}{d-i}^2 W_{i,k}^T W_{i,k}.
\end{align*}
On the other hand, $E_j$ is an eigenspace for the matrix $W_{i,k}^T W_{i,k}$ for its eigenvalue $\binom{k-j}{k-i}\binom{n-i-j}{k-i}$, which will be proved as Lemma \ref{lem_eig_W} in Appendix \ref{app}. Note that this eigenvalue vanishes whenever $j>i$. Therefore,
\begin{align*}
\vec{h}^T(W_{d,k}^T\overline{W}_{d,d}W_{d,k}) \vec{h}&=(\sum_{j=0}^k \vec{h}_j)^T  \left(\sum_{i=0}^d (-1)^i \binom{k-i}{d-i}^2 W_{i,k}^T W_{i,k}\right) (\sum_{j=0}^k \vec{h}_j)\\
&=\sum_{i=0}^d (-1)^i \binom{k-i}{d-i}^2 \left(\sum_{j=0}^i \binom{k-j}{k-i}\binom{n-i-j}{k-i}\|\vec{h}_j\|^2\right)\\
&=\sum_{j=0}^d \|\vec{h}_j\|^2 \cdot \left(\sum_{i=j}^d (-1)^i \binom{k-i}{d-i}^2 \binom{k-j}{k-i}\binom{n-i-j}{k-i}\right)\\
&=\sum_{j=0}^d(-1)^j\binom{k-j}{d-j}\binom{n-d-j}{d-j}\binom{n-d-j}{k-d} \|\vec{h}_j\|^2
\end{align*}
We leave the last equality as Lemma \ref{identity_1} whose proof is given in the appendix.\end{proof}

Now we are ready to prove our main Theorem, the $d$-degree version of the Erd\H os--Ko--Rado Theorem.

\begin{proof}[Proof of Theorem \ref{thm_main}:]
We prove by contradiction, suppose $\delta_d(\mathcal{F})>\binom{n-d-1}{k-d-1}$.

Apply Lemma \ref{lem1} and \ref{lem2} to $\mathcal{F}$ and rewrite the two inequalities obtained as:
$$0 \ge \sum_{i=0}^d a_i \|\vec{h}_i\|^2 - c |\mathcal{F}|,$$
and $$0 < \sum_{i=0}^d b_i \|\vec{h}_i\|^2 - f |\mathcal{F}| + g,$$
where $a_i=\frac{k-d}{n-k}\binom{n-k}{k}+(-1)^i\binom{n-k-i}{k-i}$, $c=\frac{k-d}{n-k}\binom{n-k}{k}$, $b_i=(-1)^i\binom{k-i}{d-i}\binom{n-d-i}{d-i}\binom{n-d-i}{k-d}$, $f=2\binom{n-d-1}{k-d-1}\binom{n-d}{d}\binom{k}{d}$, and $g=\binom{n-d-1}{k-d-1}^2\binom{n}{d}\binom{n-d}{d}$.
We subtract the second inequality by the first inequality multiplied by $b_1/a_1$. It is not hard to observe that both $b_1$ and $a_1=-\frac{d}{n-k}\binom{n-k}{k}$ are negative. Thus we have 
$$0 < \left(b_0-\frac{a_0b_1}{a_1}\right)\|\vec{h}_0\|^2 + \left(\sum_{i=2}^d (b_i-\frac{a_ib_1}{a_1})\|\vec{h}_i\|^2\right) -\left(f-\frac{cb_1}{a_1}\right)|\mathcal{F}| + g.$$
We claim the following: 
\begin{enumerate}
\item The coefficient of $\|\vec{h}_0\|^2$ is $b_0-a_0b_1/a_1=\frac{n(k-d)}{k(n-d)}\cdot b_0$, which is strictly positive.
\item For each even integer $2 \le i \le d$, we have $b_i \le a_ib_1/a_1$ as long as $n \ge 2k+1$. For each odd integer $3 \le i \le d$, we have $b_i \le a_ib_1/a_1$ as long as $n \ge 2k+2d-3$.
\item $f-b_1c/a_1=\frac{2kn-dk-dn}{k(n-d)}\cdot\binom{n-d-1}{k-d-1}\binom{n-d}{d}\binom{k}{d}.$
\end{enumerate}
Suppose all three claims are true. Note that $\|\vec{h}_i\|^2 \ge 0$, so using $\|\vec{h}_0\|^2=\langle \vec{h}_0, \frac{\vec{1}}{\sqrt{\binom{n}{k}}}\rangle^2 = |\mathcal{F}|^2/\binom{n}{k}$, we have 
\begin{align*}
0 &< \frac{n(k-d)}{k(n-d)}\cdot\binom{k}{d}\binom{n-d}{d}\binom{n-d}{k-d} \cdot \frac{|\mathcal{F}|^2}{\binom{n}{k}}\\
&~~~-\frac{2kn-dk-dn}{k(n-d)}\cdot\binom{n-d-1}{k-d-1}\binom{n-d}{d}\binom{k}{d}|\mathcal{F}| + \binom{n-d-1}{k-d-1}^2\binom{n}{d}\binom{n-d}{d}.
\end{align*}
After simplification, this quadratic inequality is equivalent to
\begin{align}\label{imp_ineq}
0 < \left(|\mathcal{F}|-\binom{n-1}{k-1}\right)\left(|\mathcal{F}|-\binom{n-d-1}{k-d-1}\frac{\binom{n}{d}}{\binom{k}{d}}\right),
\end{align}
Recall that we assume $\delta_d(\mathcal{F}) >\binom{n-d-1}{k-d-1}$. This implies 
$$|\mathcal{F}|=\frac{\sum_{S:|S|=d} d_S}{\binom{k}{d}}>\binom{n-d-1}{k-d-1}\frac{\binom{n}{d}}{\binom{k}{d}}.$$ Therefore, we must have $|\mathcal{F}|>\binom{n-1}{k-1}$, contradicting the Erd\H os--Ko--Rado Theorem. This completes the proof that $\delta_d(\mathcal{F}) \le \binom{n-d-1}{k-d-1}$ for intersecting family $\mathcal{F}$.


It remains to verify the three claims we made earlier. The first and third are done by straightforward calculations:
\begin{align*}
b_0-\frac{a_0b_1}{a_1}&=b_0\left(1-\frac{a_0}{a_1}\cdot \frac{b_1}{b_0}\right)=b_0\left(1-\frac{\frac{n-d}{n-k}\binom{n-k}{k}}{\frac{-d}{n-k}\binom{n-k}{k}}\cdot \frac{-\binom{k-1}{d-1}\binom{n-d-1}{d-1}\binom{n-d-1}{k-d}}{\binom{k}{d}\binom{n-d}{d}\binom{n-d}{k-d}}\right)\\
&=b_0\cdot \left(1- \frac{n-d}{d} \cdot \frac{d}{k}\cdot \frac{d}{n-d}\cdot \frac{n-k}{n-d}\right)=b_0 \cdot \frac{n(k-d)}{k(n-d)}.
\end{align*}
\begin{align*}
f-\frac{b_1 c}{a_1}&=2 \binom{n-d-1}{k-d-1}\binom{n-d}{d}\binom{k}{d}-\frac{\frac{k-d}{n-k}\binom{n-k}{k}\cdot\binom{k-1}{d-1}\binom{n-d-1}{d-1}\binom{n-d-1}{k-d}}{\frac{d}{n-k}\binom{n-k}{k}}\\
&=\binom{n-d-1}{k-d-1}\binom{n-d}{d}\binom{k}{d}\cdot \left(2-\frac{k-d}{d}\cdot \frac{d}{k} \cdot  \frac{d}{n-d}\cdot\frac{n-k}{k-d}\right)\\
&=\binom{n-d-1}{k-d-1}\binom{n-d}{d}\binom{k}{d}\cdot\frac{2kn-dk-dn}{k(n-d)}.
\end{align*}
To verify the second claim, since $a_1, b_1<0$, it suffice to prove for $i=2, \cdots, d$, $b_i/(-b_1) \le a_i/(-a_1)$. We plug in the values of $a_1, b_1, a_i, b_i$. Denote by $(x)_i$ the falling factorial $x(x-1)\cdots (x-i+1)$, then
$$\frac{a_i}{-a_1}=\frac{\frac{k-d}{n-k}\binom{n-k}{k}+(-1)^i\binom{n-k-i}{k-i}}{\frac{d}{n-k}\binom{n-k}{k}}=\frac{k-d+(-1)^i \cdot k \cdot \frac{(k-1)_{i-1}}{(n-k-1)_{i-1}}}{d}.$$
$$\frac{b_i}{-b_1}=(-1)^i \cdot \frac{\binom{k-i}{d-i}\binom{n-d-i}{d-i}\binom{n-d-i}{k-d}}{\binom{k-1}{d-1}\binom{n-d-1}{d-1}\binom{n-d-1}{k-d}}=(-1)^i \cdot \frac{[(d-1)_{i-1}]^2\cdot (n-k-1)_{i-1}}{[(n-d-1)_{i-1}]^2\cdot(k-1)_{i-1}}.$$
Therefore the inequality $b_i/(-b_1) \le a_i/(-a_1)$ we would like to establish can be rewritten as
$$(-1)^i \left[ d \cdot \frac{[(d-1)_{i-1}]^2\cdot (n-k-1)_{i-1}}{[(n-d-1)_{i-1}]^2\cdot(k-1)_{i-1}}- k \cdot \frac{(k-1)_{i-1}}{(n-k-1)_{i-1}} \right] \le k-d.$$
We define
$$S_i(n)=\frac{(k-1)_{i-1}}{(n-k-1)_{i-1}},~~~ T_i(n)=\frac{[(d-1)_{i-1}]^2\cdot (n-k-1)_{i-1}}{[(n-d-1)_{i-1}]^2\cdot(k-1)_{i-1}}.$$ 
Since $k>d$, we have $(d-1)_{i-1} < (k-1)_{i-1}$ and $(n-d-1)_{i-1}>(n-k-1)_{i-1}$. It is not hard to see that $S_i(n)> T_i(n)$ for $i \ge 2$. This implies that when $i$ is even, the left hand side is negative, so this inequality is automatically true in that case. We remark that this simple observation already shows the $2$-degree Erd\H os--Ko--Rado Theorem holds for every $n \ge 2k+1$.

Now suppose $i$ is odd and thus $i \ge 3$. We would need to show that $$k\cdot S_i(n)-d \cdot T_i(n) \le k-d.$$ 
The proof of this last inequality is somewhat technical, so we leave it to the appendix as Lemma \ref{lemma_technical}.
\end{proof}


\section{Concluding Remarks}\label{sec_concluding}
In this paper, we show that the $d$-degree analogue of the Erd\H os--Ko--Rado Theorem holds for $n \ge 2k+2d-3$. On the other hand, we are not aware of any construction strictly beating the $1$-star construction, in the range of $n \ge 2k+1$. In \cite{degree_ekr}, it was shown that for every $n \ge 2k+1$, the $1$-star is the unique example with $\delta_1=\binom{n-2}{k-2}$. However this fails to be the case for larger $d$. The simplest example is to take a Fano plane for $n=7$ and $k=3$. The minimum $2$-degree is equal to $1$, while the $1$-star also has minimum degree $\binom{n-3}{k-3}=1$ in this case. For the next case $(n,k,d)=(9,4,2)$, a more sophisticated construction  by \"Osterg\aa rd \cite{ostergard} gives an intersecting family $\mathcal{F} \subset \binom{[9]}{4}$ that has $\delta_2(\mathcal{F})=6=\binom{n-3}{k-3}$. This example is obtained from gluing three copies of the $2$-$(6,3,2)$ design mentioned in the introduction. More discussions on the $n=2k+1$ case can be found in \cite{ihr-kup}.

Even though unlike the $d=1$ case, the uniqueness no longer holds for $d=2$ and $n=2k+1$, the following conjecture is still plausible and has been now verified for $d \in \{0, 1, 2\}$. 
\begin{conjecture}\label{main_conj}
For every $k>d \ge 0$ and $n \ge 2k+1$, an intersecting family $\mathcal{F} \subset \binom{[n]}{k}$ always has minimum $d$-degree 
$$\delta_d(\mathcal{F}) \le \binom{n-d-1}{k-d-1}.$$
\end{conjecture}

Here let we mention one possible approach to further lower the threshold. Recall that Lemma \ref{lem1} works even if we replace $\frac{k-d}{n-k}\binom{n-k}{k}$ by any constant at least $\binom{n-k-d-1}{k-d-1}$. It is possible that by selecting a different constant in Lemma \ref{lem1}, after subtracting the two inequalities given by Lemma \ref{lem1} and Lemma \ref{lem2} in the proof of Theorem \ref{thm_main}, one could have the coefficient of not just $\|\vec{h}_1\|^2$, but also that of $\|\vec{h}_3\|^2$ to vanish, while still keeping all the other coefficients of $\|\vec{h}_i\|^2$ non-positive. Such more careful analysis may further lower the threshold needed for $n$, possibly solving for the range $n \ge 2k+cd$ for a smaller $c$. But we believe that  if Conjecture \ref{main_conj} is indeed true, some new ideas are necessary to cover the whole range $n \ge 2k+1$.

One can also ask a more general question: for a given family $\mathcal{G} \subset 2^{[n]}$, when can we guarantee that any intersecting subfamily $\mathcal{F} \subset \mathcal{G}$ has $\delta_d(\mathcal{F}) \le \max_{i \in [n]} \delta_d(\mathcal{G}_i)$? Here $\mathcal{G}_i$ is the family of subsets of $\mathcal{G}$ that contains $i$.  Theorem \ref{thm_main} shows that $\mathcal{G}=\binom{[n]}{k}$ has such property for $d \ge 2$ and $n \ge 2k+2d-3$. On the other hand for $d=0$, Chv\'atal \cite{chvatal_conj} conjectured that as long as $\mathcal{G}$ is {\it hereditary}, meaning that whenever $A \subset B$ and $B \in \mathcal{G}$ then $A \in \mathcal{G}$, then such inequality holds. One might want to guess that this also works for $d \ge 1$. However we can just take $\mathcal{G}=2^{[n]}$ and $\mathcal{F}=\binom{[n]}{\ge (n+1)/2}$ for odd $n$. It is not hard to see that $\delta_d(\mathcal{F}) > \delta_d(\mathcal{G}_i)$ for every $i \in [n]$. 
It would be very interesting to find the right assumptions that may generalize Chv\'atal's conjecture to the $d$-degree case.

\appendix

\section{Technical lemmas}\label{app}
In this appendix, we give the proofs of several technical lemmas used in the paper. Recall that the $(i, j)$-inclusion matrix $W_{i, j}$ is the matrix with rows indexed by all $i$-subsets of $[n]$, columns indexed by all $j$-subsets of $[n]$, and $(W_{i,j})_{S,T}=1$ if $S \subset T$, and $0$ otherwise. And the $(i,j)$-disjointness matrix $\overline{W}_{i, j}$ has $(S, T)$-entry $1$ if $S \cap T=\emptyset$ and $0$ otherwise. It is not hard to see that they have the same row space.
\begin{lemma}\label{lem_eig_W}
Take any column vector $\vec{v} \in E_j \subset \mathbb{R}^{\binom{n}{k}}$, then for $j>i$ we have $W_{i,k}^T W_{i,k} \vec{v}=0$, and for $0 \le j \le i$, 
$$W_{i,k}^T W_{i,k} \vec{v}=\binom{k-j}{k-i}\binom{n-i-j}{k-i}\vec{v}$$
\end{lemma}
\begin{proof}
As discussed in Section \ref{sec_pre}, $E_j$ is the orthogonal complement of $F_{j-1}$ in $F_j$. And $F_j$ is spanned by the indicating vectors of all the $j$-stars. Therefore, $W_{j-1, k}\vec{v}=0$, and $\vec{v}=W_{j,k}^T \vec{u}$ for some $\vec{u} \in \mathbb{R}^{\binom{n}{j}}.$
Therefore for $j>i$,
$$W_{i,k}\vec{v}=\frac{W_{i,j-1}W_{j-1,k}}{\binom{k-i}{j-1-i}}\cdot \vec{v}=\vec{0}.$$

For $\ell<j$, note that $W_{\ell, k}$ and $\overline{W}_{\ell, k}$ have the same row spaces. Since $\vec{v}$ is in the orthogonal complement of $F_\ell$, we know that $W_{\ell, k}\vec{v}=\vec{0}$ and thus we also have $\overline{W}_{\ell, k}\vec{v}=\vec{0}$. Using $\vec{v}=W_{j,k}^T \vec{u}$, we have
$$\vec{0}=\overline{W}_{\ell, k}W_{j,k}^T \vec{u}=\binom{n-\ell-j}{k-j}\overline{W}_{\ell, j}\vec{u}.$$
Again since $W_{\ell, j}$ and $\overline{W}_{\ell, j}$ have the same row spaces, we have $W_{\ell, j}\vec{u}=\vec{0}$ whenever $\ell<j$.

Now for $j \le  i$, we apply the following identity (6.11 from Page 124 of \cite{godsil-meagher}):
\begin{align*}
W_{i,k}W_{j,k}^T = \sum_{\ell=0}^j \binom{n-i-j}{n-k-\ell} W_{\ell, i}^TW_{\ell, j}.
\end{align*}
From the above discussions, $W_{\ell, j}\vec{u}=0$ whenever $\ell<j$, therefore
\begin{align*}
W_{i,k}^T W_{i,k} \vec{v} &= W_{i,k}^T W_{i,k} W_{j,k}^T \vec{u} = \sum_{\ell=0}^j \binom{n-i-j}{n-k-\ell} W_{i,k}^T W_{\ell, i}^TW_{\ell, j}\vec{u}\\
&=\sum_{\ell=0}^j \binom{n-i-j}{n-k-\ell} \binom{k-\ell}{k-i} W_{\ell, k}^TW_{\ell, j}\vec{u}\\
&=\binom{n-i-j}{n-k-j} \binom{k-j}{k-i} W_{j, k}^TW_{j,j}\vec{u}=\binom{n-i-j}{n-k-j} \binom{k-j}{k-i} \vec{v},
\end{align*}
which completes the proof of this lemma.
\end{proof}
\begin{lemma}\label{identity_1}
For integers $j \le d < k \le n/2$,we have
$$\sum_{i=j}^d (-1)^i \binom{k-i}{d-i}^2 \binom{k-j}{k-i}\binom{n-i-j}{k-i}=(-1)^j\binom{k-j}{d-j}\binom{n-d-j}{d-j}\binom{n-d-j}{k-d}.$$
\end{lemma}
\begin{proof}
We divide the left hand side by the right hand side. After canceling terms, the identity we would like to prove is equivalent to 
$$\sum_{i=j}^d (-1)^{i-j} \frac{\binom{d-j}{i-j}\binom{n-i-j}{d-i}}{\binom{n-d-j}{n-2d}}=1.$$
This identity follows from  comparing the coefficients of $x^{d-j}=x^{i-j}\cdot x^{d-i}$ in the expansions of the following two equal generating functions at $x=0$:
$$(1-x)^{d-j} \cdot (1-x)^{-(n-d-j+1)} = (1-x)^{-(n-2d+1)}.$$
\end{proof}
\begin{lemma}\label{lemma_technical}
Suppose $k > d \ge i \ge 3$ and $n \ge 2k+2d-3$, let $$S_i(n)=\frac{(k-1)_{i-1}}{(n-k-1)_{i-1}},~~~~T_i(n)=\frac{[(d-1)_{i-1}]^2\cdot (n-k-1)_{i-1}}{[(n-d-1)_{i-1}]^2\cdot(k-1)_{i-1}},$$ 
where $(x)_i$ is the falling factorial $(x)_i=x(x-1)\cdots (x-i+1)$.
Then the following inequality holds: $$k\cdot S_i(n)-d \cdot T_i(n) \le k-d.$$ 
\end{lemma}
\begin{proof}
    It is not hard to see that $0 < T_i(n)<S_i(n) \le 1$ since $n > 2k$ and $k>d$. Furthermore, since $d<k$, we have
$$\frac{S_i(n+1)}{S_i(n)}=\frac{n-k-i+1}{n-k} \le \left(\frac{n-d-i+1}{n-d}\right)^2\cdot \frac{n-k}{n-k-i+1}=\frac{T_i(n+1)}{T_i(n)}.$$
Thus suppose $k\cdot S_i(n)-d \cdot T_i(n) \le k-d$, we can immediately get
$$k\cdot S_i(n+1)-d\cdot T_i(n+1) \le (k\cdot S_i(n)-d\cdot T_i(n))\cdot \frac{S_{i}(n+1)}{S_i(n)} \le k\cdot S_i(n)-d\cdot T_i(n) \le k-d.$$
Therefore for fixed parameters $k, d$ and $i$, it suffices to verify the inequality $k\cdot S_i(n)-d \cdot T_i(n) \le k-d$ for the base case $n=2k+2d-3$, and then the proof for larger $n$ follows by induction.

Fix $k, d$ and $n=2k+2d-3$, we let 
$$\alpha_i=\frac{(k-1)_{i-1}}{(n-k-1)_{i-1}},~~~~\beta_i=\frac{(d-1)_{i-1}}{(n-d-1)_{i-1}}.$$ 
Then $S_i=\alpha_i$, $T_i=\beta_i^2/\alpha_i$. We first show that $kS_i-dT_i$ decreases in $i$. First note that $\alpha_{i+1}/\alpha_i=(k-i)/(n-k-i)<1$, and similarly $\beta_{i+1}/\beta_i<1$. So $\{\alpha_i\}$ and $\{\beta_i\}$ are both decreasing in $i$. Furthermore,
$$\alpha_i-\alpha_{i+1}=\alpha_i(1-\frac{k-i}{n-k-i})=\frac{n-2k}{n-k-i} \cdot \alpha_i=\frac{(k-1)_{i-1}(n-2k)}{(n-k-1)_i},$$
and similarly $\beta_i-\beta_{i+1}=\frac{(d-1)_{i-1}(n-2d)}{(n-d-1)_i}.$ 
Since $i\ge 3$, it is straightforward to check that $(n-2k)(k-1)(k-2) \ge (n-2d)(d-1)(d-2)$ for $n=2k+2d-3$. Therefore using $k>d$ again, we have
\begin{align*}
\alpha_i-\alpha_{i+1}&=\frac{(k-3)_{i-3}}{(n-k-1)_i} \cdot (k-1)(k-2)(n-2k)\\
& \ge \frac{(d-3)_{i-3}}{(n-d-1)_i} \cdot (d-1)(d-2)(n-2d) =\beta_i-\beta_{i+1}.
\end{align*}
This shows $\alpha_i-\beta_i$ decreases in $i$. On the other hand, 
$$\frac{\beta_{i+1}/\alpha_{i+1}}{\beta_{i}/\alpha_{i}}=\frac{(d-i)/(n-d-i)}{(k-i)/(n-k-i)}<1,$$
so $\beta_i/\alpha_i$ also decreases in $i$. We immediately know that 
$$k\cdot S_i - d \cdot T_i=(k-d)S_i + d(S_i-T_i) = (k-d)\alpha_i + d(\alpha_i-\beta_i)(1+\frac{\beta_i}{\alpha_i})$$
also decreases in $i$. Thus to show $kS_i-dT_i \le k-d$, it suffices to check for $i=3$. For this case, the inequality we would like to prove is equivalent to 
$$d(\alpha_3-\beta_3)(1+\frac{\beta_3}{\alpha_3}) \le (k-d)(1-\alpha_3).$$
Calculations give that
$$1-\alpha_3=1-\frac{(k-1)(k-2)}{(n-k-1)(n-k-2)}=\frac{(n-3)(n-2k)}{(n-k-1)(n-k-2)}.$$
$$\alpha_3-\beta_3=\frac{(k-d)(n-3)((k+d-3)n-2dk+4)}{(n-k-1)(n-k-2)(n-d-1)(n-d-2)}.$$
\begin{align*}
1+\frac{\beta_3}{\alpha_3}&=1+\frac{(d-1)(d-2)/(n-d-1)(n-d-2)}{(k-1)(k-2)/(n-k-1)(n-k-2)}\\
& \le 1+\frac{(d-1)(d-2)}{(k-1)(k-2)} \\
&\le 
 1+\frac{((k-1)-1)((k-1)-2)}{(k-1)(k-2)}\\
& = 1+\frac{k-3}{k-1}=\frac{2k-4}{k-1}.
\end{align*}
Combining these observations, it suffices to prove for $n=2k+2d-3$,
$$\frac{d(2k-4)((k+d-3)n-2dk+4)}{(k-1)(n-d-1)(n-d-2)} \le n-2k$$
Plugging in $n=2k+2d-3,$ it becomes
$$(4d-12)k^3+(4d^2-30d+66)k^2-(2d^3+3d^2-53d+114)k+(6d^3-15d^2-15d+60) \ge 0$$
Since $d \ge 3$, it is easy to verify that $4d-12 \ge 0$, $4d^2-30d+66 \ge 0$. Thus using $k \ge d+1$, the left hand side is at least
\begin{align*}
&k((4d-12)(d+1)^2+(4d^2-30d+66)(d+1)-2d^3-3d^2+53d-114)+(6d^3-15d^2-15d+60)\\
&=(6d^3-33d^2+69d-60)k+(6d^3-15d^2-15d+60)
\end{align*}
It is not hard to check that when $d \ge 3$, both terms in the expression above are again non-negative. This completes the proof.
\end{proof}

\end{document}